\RequirePackage{fix-cm}
\documentclass[smallextended,envcountsect,envcountsame]{svjour3hack}       

\smartqed  
\usepackage{graphicx}

\journalname{}

\usepackage{amsmath, amssymb, enumitem,xcolor}
\usepackage[utf8]{inputenc}
\usepackage{xurl}
\usepackage{float}

\newcommand{\rank}{\operatorname{rank}}

\newcommand{\abs}[1]{\vert#1\vert}

\makeatletter
\renewcommand*{\top}{%
  {\mathpalette\@transpose{}}%
}
\newcommand*{\@transpose}[2]{%
  \scriptsize
  \raisebox{\depth}{$\m@th#1\mathsf{T}$}%
}
\makeatother

\usepackage[hypertexnames=false,colorlinks=true,breaklinks=true,bookmarks=true,urlcolor=blue,citecolor=blue,linkcolor=blue,bookmarksopen=false,draft=false]{hyperref}

\makeatletter
\let\c@prop\c@theorem
\let\c@cor\c@theorem
\let\c@lemma\c@theorem
\let\c@definition\c@theorem
\let\c@example\c@theorem
\let\c@remark\c@theorem
\let\c@obs\c@theorem
\let\c@claim\c@theorem
\makeatother

\newcommand{\R}{\mathbb{R}}

\newcommand{\bA}{\mathbf{A}}
\newcommand{\bB}{\mathbf{B}}

\newcommand{\bb}{\mathbf{b}}
\newcommand{\bc}{\mathbf{c}}
\newcommand{\be}{\mathbf{e}}
\newcommand{\bg}{\mathbf{g}}
\newcommand{\bv}{\mathbf{v}}

\newcommand{\bw}{\mathbf{w}}
\newcommand{\bx}{\mathbf{x}}
\newcommand{\by}{\mathbf{y}}

\DeclareMathOperator{\vol}{vol}
\DeclareMathOperator{\convcl}{convcl}

\DeclareMathOperator{\conv}{conv}

\allowdisplaybreaks[1]

\begin{document}

\title{Gaining or Losing Perspective for Convex Multivariate Functions on Box Domains\thanks{This work was supported in part by ONR grant N00014-21-1-2135.
This work is partially based upon work supported by the
National Science Foundation under Grant No. DMS-1929284 while
the authors were in residence at the Institute for Computational and Experimental Research in
Mathematics in Providence, RI, during the Discrete Optimization program.
}
}


\titlerunning{Gaining or Losing Perspective for Convex Multivariate Functions}        

\author{Luze Xu         \and
        Jon Lee 
}


\institute{Luze Xu \at
              Department of Mathematics, University of California at Davis\\
              \email{lzxu@ucdavis.edu}           
         \and
         Corresponding Author: Jon Lee \at
          Department of Industrial and Operations Engineering, University of Michigan\\
          \email{jonxlee@umich.edu}
}

\date{February 8, 2023. Revised, November 8, 2023. Revised March 31, 2024.}

\maketitle

\begin{abstract}

MINLO (mixed-integer nonlinear optimization) formulations of the disjunction between the origin and a polytope via a binary indicator variable is broadly used in nonlinear combinatorial optimization for modeling a fixed cost associated with carrying out a group of activities and a convex cost function associated with the levels of the activities.
The perspective relaxation of such models is often used to solve to global optimality in a branch-and-bound context, but it typically requires suitable conic solvers and is not compatible with general-purpose NLP software in the presence of other classes of constraints.
This motivates the investigation of when simpler but weaker relaxations may be adequate.
Comparing the volume (i.e., Lebesgue measure) of the relaxations as a measure of tightness, we lift some of the results related to the simplex case to the box case.
In order to compare the volumes of different relaxations in the box case, it is necessary to find an appropriate concave upper bound that preserves the convexity and is minimal, which is more difficult than in the simplex case. To address the challenge beyond the simplex case, the triangulation approach is used.


\end{abstract}


\section{Introduction}

We consider a common MINLO
(mixed-integer nonlinear optimization) modeling situation: when an
indicator variable $z$ is
equal to 0, a vector $\bx$ of $d$ decision variables is set to $\mathbf{0}\in\mathbb{R}^d$,
and when $z=1$, the vector $\bx$ is required to be in a polytope $Q\subset \mathbb{R}^d$
(see, for example, \cite{gunlind1,lee_gaining_2020} and the references therein).
In applications, when $z=1$, there is typically a fixed cost $c$ incurred and
a variable convex cost $f(\bx)$.
The tightest formulation of a natural disjunctive model uses the well-known
perspective relaxation. But, generally, such a relaxation
is best handled with
conic solvers (e.g., \verb;MOSEK;  \cite{MOSEK} and \verb;SDPT3; \cite{SDPT3}); see \cite[Sec. 1.2]{lee_gaining_2020}.
Even then, such solvers are not generally capable of handling all possible convex $f$
(at present, most can handle exponential functions and convex power functions).
There is significant value in understanding when (i.e., for which $f$ and $Q$) lighter relaxations than the
perspective relaxation can be successfully employed.

\paragraph{Key definitions.}
In what follows, we have a convex function $f:\R^d_{\geq 0}\rightarrow \R$
and a polytope $Q \subset \R^d_{\geq 0}\setminus \{\mathbf{0}\}$.
We define the disjunctive set
\begin{align*}
& D(f,Q) := \left\{(\mathbf{0}_d,0,0)\right\}\cup D_1(f,Q), \mbox{ where} \\
& D_1(f,Q) : = \left\{(\bx,y,1)\in \R^d\times \R \times \{0,1\} ~:~
\mu(\bx) \geq y \geq f(\bx), ~\bx\in Q\right\},
\end{align*}
and $\mu$
is a concave upper bound (function) of $f$ on $Q$.
The disjunction models the choice of either $(\bx,y) =(\mathbf{0}_d,0)$
or $(\bx,y)$ is above the graph of $f$ and below the graph of
$\mu$ on the domain $Q$. The disjunction does this
through the indicator variable $z\in\{0,1\}$.
 Note that when $Q$ is
a $d$-simplex, say $Q:=\conv\{\bv^0,\ldots,\bv^d\}$, then the best choice of $\mu$ is
the unique linear function that passes through the $d+1$ points $(\bv^i,f(\bv^i))$. But in the general case, when $Q$ is given by a linear-inequality description (as is typical),
it may  not be practical to calculate the
best  concave upper bound of $f$ on $Q$ (unless $d$ is quite small).
So understanding how convex relaxations of $D(f,Q)$ depend on the
chosen $\mu$ is quite important.

The \emph{perspective function}, $\tilde{f}:\R^{d+1}\rightarrow \R$, defined by
\[
\tilde{f}(\bx,z):= \left\{
          \begin{array}{ll}
             z f(\bx/z), & \hbox{for $z>0$;} \\
             +\infty, & \hbox{otherwise}
          \end{array}
         \right.
\]
is well-known to be a convex function, when $f:~\R^d \rightarrow \R$ is convex  (see \cite[Sec. IV.2.2, Page 160]{perspecbook}).
If we evaluate the closure of $\tilde{f}$ (whose epigraph is the closure of the epigraph of $\tilde{f}$, see \cite[Sec. IV, Definition 1.2.4, Page 149]{perspecbook})
at $(\mathbf{0}_d,0)$, we get $0$ (see \cite[Sec. IV, Remark 2.2.3, Page 162]{perspecbook}).
So we can define the \emph{perspective relaxation}
\begin{align*}
P(f,Q):=&\convcl \left\{ (\bx,y,z) \in \R^d\times \R \times \R ~:~ \right.\\
&\left. \qquad \tilde{\mu}(\bx,z) \geq y \geq \tilde{f}(\bx,z), ~\bx\in z\cdot Q,~ 1\geq z > 0
\vphantom{\R^d}\right\},
\end{align*}
where the upper bound $\tilde{\mu}(\bx,z)$ is the negative of the perspective function of $-\mu(\bx)$, and is thus concave itself
(and even linear when $\mu$ is), and $\convcl$ denotes convex closure.

Additionally, we consider the following special case: the domain of the convex function $f$ is $\conv(Q\cup \{\mathbf{0}\})=\{z\cdot Q ~:~ 0\le z\le 1\}$, and $f(\mathbf{0})=0$. We can then define the \emph{na\"{i}ve relaxation}:
\begin{align*}
P^0(f,Q):=&\left\{ (\bx,y,z) \in \R^d\times \R \times \R ~:~\right.\\
&\qquad\left.\tilde{\mu}(\bx,z) \geq y \geq f(\bx), ~\bx\in z\cdot Q,~ 1\geq z \geq 0 \right\}.
\end{align*}
Note that $P^0(f,Q)$ is easier to work with than $P(f,Q)$,
computationally, because
it does not employ a perspective function.

Finally, to compare   $P^0:=P^0(f,Q)$ and $P:=P(f,Q)$,
we define the volume (i.e., Lebesgue measure) of a polyhedron $P$ (resp. $P^0$) as $\vol(P)$ (resp. $\vol(P^0)$), the \emph{cut-off amount}  $\Delta(P^0,P)=
\vol(P^0) - \vol(P)$,
and the \emph{cut-off ratio}  $\frac{\Delta(P^0,P)}{\vol(P^0)}$.

\paragraph{Background.} \cite{LM1994} introduced the idea of using $d$-dimensional volume
(i.e., Lesbesgue measure) to compare convex relaxations of nonconvex sets in $\mathbb{R}^d$,
in the context of global optimization for combinatorial-optimixation problems modeled as integer programs;
also see
\cite{LeeSkipperSpeakmanMPB2018}
and the many references therein.
\cite{PerspectiveWCGO,lee_gaining_2020}
applied this idea to the perspective
(convex-hull) and na\"{\i}ve  relaxations of a disjunction related to convex univariate functions (i.e., $d=1$) on an interval domain.
\cite{PLPerspecCTW,lee2020piecewise}
extended the analysis to piecewise-linear underestimation of the convex function.
\cite{xuleemulti} generalized the ideas of \cite{PerspectiveWCGO,lee_gaining_2020} to multivariate convex functions on a simplex domain. Our goal in what
follows is to push the multivariate ideas further, to handle box and some zonotope domains.
We note that such domains are very natural for applications,
and they are precisely the basic domains focused on in
all state-of-the-art software for nonconvex MINLO, which are based on ``spatial branch-and-bound''; see \cite{TawarSahinBook,Baron1996}, for example.

\paragraph{Organization and contributions.}
In Section \ref{sec:basics}, we derive general formulae for the volumes of perspective and na\"{i}ve relaxations which are dependent on the upper bound function $\mu$, and we demonstrate that the  cut-off amount is independent of $\mu$.
We also establish that the volume formulae for homogeneous functions
given by \cite{xuleemulti}
for simplices can be generalized to all polytopes.
In Section \ref{sec:compare}, we derive the formula for the
cut-off amount for two natural families of functions, extending the corresponding result in \cite{xuleemulti} to the box case.
As in \cite{xuleemulti}, we are interested in analyzing the asymptotic behavior of the cut-off ratio  to compare perspective and na\"{i}ve relaxations.
One major difficulty for the analysis of the extension to  boxes is that the best concave  upper bound $\mu(\bx)$ is generally no longer a hyperplane agreeing with $f(x)$ at all the vertices, as in the simplex case.
In Section \ref{sec:cut_off}, we analyze the cut-off ratio for two natural upper bounds: a constant bound and the concave envelope bound (i.e., the smallest concave upper bound), and show that they have similar asymptotic behavior.
In Section \ref{sec:conc}, we make some brief conclusions.

\paragraph{Notation.}
In what follows, we use boldface lower-case for vectors and boldface upper-case for matrices. $\displaystyle\be_n^{d}$ denotes the $n$-th unit vector in $\R^{d}$, and the superscript $d$ is often dropped if the dimension is clear from the context.
$\mathbf{I}_d$ is the order-$d$ identity matrix.
$Q_u:=\bv_0+u[0,1]^d$ $(\bv_0>0, u>0)$ denotes the box obtained from a unit hypercube scaled by $u$ and translated by $\bv_0$.

\section{Basics}\label{sec:basics}
Generalizing from the simplex case \cite[Section 2]{xuleemulti},
we have the following formulae for the
volume of $P(f,Q)$ and $P^0(f,Q)$.
\begin{theorem}\label{thm:perspecvol}
    Suppose that $f$ is a continuous and convex function on the polytope $Q \subset \R^d_{\geq 0}\setminus \{ \mathbf{0}\}$. Then
    $$
    \vol(P(f,Q)) = \frac{1}{d+2}\int_{Q}\mu(\bx) d\bx - \frac{1}{d+2}\int_{Q} f(\bx)d\bx .
    $$
\end{theorem}
\begin{proof}
    Considering the definition of the perspective relaxation, its volume is
    \begin{align*}
    \vol(P(f,Q)) &= \int_{\bx\in z\cdot Q,0\le z\le 1} (z\mu(\bx/z) - zf(\bx/z))d\bx dz \\
    &= \int_{0\le z\le 1}z^d\int_{\tilde{\bx}\in Q}(z\mu(\tilde{\bx}) - zf(\tilde{\bx})) d\tilde{\bx} dz\\
    &= \int_{0\le z\le 1}z^{d+1}dz\int_{\tilde{\bx}\in Q}(\mu(\tilde{\bx}) -f(\tilde{\bx})) d\tilde{\bx}\\
    &=\frac{1}{(d+2)}\int_{Q}(\mu(\bx)-f(\bx)) d\bx. \tag*{\qed}
    \end{align*}
\end{proof}

Theorem \ref{thm:perspecvol} reduces the calculation of $\vol(P(f,Q))$  to the calculation of the integral $\int_{Q} \mu(\bx)d\bx$ and $\int_{Q} f(\bx)d\bx$.
We focus on computing $\int_{Q} f(\bx)d\bx$ in Sections \ref{sec:basics}, \ref{sec:compare}, and we discuss $\int_{Q} \mu(\bx)d\bx$ in Section \ref{sec:cut_off} to compare different upper bounds $\mu$.
We note that the calculation of $\vol(P^0(f,Q))$  generally requires one more layer of integration (see Theorem \ref{thm:naivevol}), while the
case of homogeneous $f$ does not (see Subsection \ref{sec:homog}).
\begin{theorem}\label{thm:naivevol}
    Suppose that $f(\mathbf{0})=0$ and $f$ is continuous and convex on $\conv(Q\cup \{\mathbf{0}\})$, where the polytope $Q\subset \R^d_{\geq 0}\setminus \{ \mathbf{0}\}$. Then
 \begin{align*}
    \vol(P^0(f,Q))
    &=\frac{1}{(d+2)}\int_{Q}\mu(\bx)d\bx - \int_0^1 z^d\int_{Q} f(z\bx) d\bx dz.
    \end{align*}
\end{theorem}
\begin{proof}
   Considering the definition of the na\"{i}ve relaxation, its volume is
    \begin{align*}
    \vol(P^0(f,Q)) &= \int_{\bx\in z\cdot Q,0\le z\le 1} (z\mu(\bx/z) - f(\bx))d\bx dz \\
    &= \int_{0\le z\le 1}z^d\int_{\tilde{\bx}\in Q}(z\mu(\tilde{\bx}) - f(z\tilde{\bx})) d\tilde{\bx} dz\\
    &= \int_{0\le z\le 1}z^{d+1}dz\int_{\tilde{\bx}\in Q}\mu(\tilde{\bx}) d\tilde{\bx} - \int_0^1 z^d\int_{Q} f(z\bx) d\bx dz\\
    &=\frac{1}{(d+2)}\int_{Q}\mu(\bx)d\bx - \int_0^1 z^d\int_{Q} f(z\bx) d\bx dz. \tag*{\qed}
    \end{align*}
\end{proof}

\begin{corollary}\label{cor:diff}
    Suppose that $f(\mathbf{0})=0$ and $f$ is continuous and convex on $\conv(Q\cup \{\mathbf{0}\})$, where the polytope $Q\subset \R^d_{\geq 0}\setminus \{ \mathbf{0}\}$. Then
 \begin{align*}
    \Delta(P^0,P)
    &=\frac{1}{(d+2)}\int_{Q}f(\bx)d\bx - \int_0^1 z^d\int_{Q} f(z\bx) d\bx dz.
    \end{align*}
\end{corollary}
\begin{remark}
$\Delta(P^0,P)$ is independent of the concave upper bound $\mu$.
\end{remark}

\subsection{\texorpdfstring{$q$}{q}-homogeneous functions}\label{sec:homog}

Suppose that $f(\bx)$ is $q$-homogeneous, i.e., $f(\lambda \bx) =\lambda^q f(\bx)$ for $\lambda\ge 0$ ($\lambda=0$ implies that $f(\mathbf{0}) = 0$). Then, for a general polytope $Q$, we can compute the volume of the na\"{i}ve relaxation and
compare it to that of the perspective relaxation.
\begin{lemma}\label{lem:naivevol_qhomo}
    Suppose that $f$ is $q$-homogeneous ($q\ge 1$) and convex on $\conv(Q\cup \{\mathbf{0}\})$, where the polytope $Q\subset \R^d_{\geq 0}\setminus \{ \mathbf{0}\}$. Then
 \begin{align*}
    \int_0^1 z^d\int_Q f(z\bx)d\bx dz &=\frac{1}{q+d+1}\int_{Q} f(\bx) d\bx.
    \end{align*}
\end{lemma}
\begin{proof}
    Because $f(\bx)$ is $q$-homogeneous, we have $f(z\bx) = z^q f(\bx)$, and we obtain
    \begin{align*}
    &
    \int_0^1 z^d\int_Q f(z\bx)d\bx dz = \int_{0}^1 z^{q+d} dz\int_{\bx\in Q} f(\bx) d\bx = \frac{1}{q+d+1}\int_{Q} f(\bx) d\bx.
    \end{align*}
    The result follows.
\qed \end{proof}

\begin{corollary}\label{cor:homo_delta}
Suppose that $f$ is $q$-homogeneous ($q\ge 1$) and convex on $\conv(Q\cup \{\mathbf{0}\})$, where the polytope $Q\subset \R^d_{\geq 0}\setminus \{ \mathbf{0}\}$. Then
\begin{align*}
  & \Delta(P^0,P)=\frac{q-1}{(d+2)(q+d+1)}\int_Q f(\bx)d\bx.
\end{align*}
\end{corollary}
\begin{proof}
It follows from Corollary \ref{cor:diff} and Lemma \ref{lem:naivevol_qhomo}.
\qed\end{proof}

\section{Comparing relaxations on some zonotope domains}\label{sec:compare}
In this section, we consider the zonotope case $Q:=
\{\bx=\bA\by+\bb\in\R^d ~:~ \by\in[0,1]^n\}$,
where $\bA\in\R^{d\times n}$, $\rank \bA =n\le d$, which includes the box case when $n=d$.
We derive formulae for the integral $\int_Q f(\bx) d\bx$ and $\Delta(P^0,P)$ for two natural families of convex functions: power of linear forms $f(\bx):=(\bc^\top\bx)^q$ ($q\ge1$) and exponential function $f(\bx):=e^{\bc^\top\bx}-1$.
The exponential function is often used to model probabilities in various contexts, such as Koopman's theory of search (e.g., \cite{koopman1953optimum}) and resource allocation (e.g., \cite{patriksson2008survey}). It is also a key component of geometric programming (e.g., \cite{serrano2015algorithms}). Additionally, in many applications of regression, the dependent variable is logarithmically transformed (e.g., financial data) after an appropriate shift, to achieve a linear model; so the linearized regression model $\log(y+1)=\bc^\top \bx$ is equivalent to
$y=f(x):=e^{\bc^\top \bx} -1$ (the shift forces the curve to go through $(\mathbf{0},0)$).
Powers of linear forms can represent a larger family of functions. For example, a polynomial can be decomposed into the sum of powers of linear forms (see \cite{baldoni_how_2010}). Alexander and Hirschowitz answered the generic problem of finding a decomposition with the smallest possible number of summands, which is the well-known polynomial Waring problem (see \cite{MR1311347}). Finally, as we motivated in our study of the simplex case, these two families of convex functions are natural generalizations of the ubiquitous univariate function $x^q$ ($q\geq 1$) and $e^x-1$.
\subsection{Power of linear forms}
We have the following result to compute $\int_Q f(\bx) d\bx$ for $f(\bx)=(\bc^\top\bx)^q$ ($q\ge1$).

\begin{proposition}\label{prop:power_linear}
For $Q:=
\{\bx=\bA\by+\bb\in\R^d ~:~ \by\in[0,1]^n\}$,
where $\bA\in\R^{d\times n}$, $\rank \bA =n\le d$, and $f(\bx) := (\bc^\top \bx)^q$ ($q\ge 1$), we have
\begin{align*}
    \int_{Q} f(\bx) d\bx  &= \sqrt{\det(\bA^\top \bA)}\int_{[0,1]^n}(\bc^\top \bA \by + \bc^\top\bb)^q d\by.
\end{align*}
In particular, when $Q:=Q_u$ ($\bA=u\mathbf{I}_d$, $\bb=\bv_0$), a unit hypercube scaled by $u$ and translated by $\bv_0$, then
\begin{align*}
    \int_{Q} f(\bx) d\bx &= u^{q+d}\int_{[0,1]^d}\left(\bc^\top\by+\frac{\bc^\top\bv_0}{u}\right)^q d\by.
\end{align*}
\end{proposition}

\begin{proof}
By affine transformation of the variables, we have
\begin{align*}
    \int_{Q} f(\bx) d\bx  &= \sqrt{\det(\bA^\top \bA)}\int_{[0,1]^n} f(\bA \by+\bb)d\by\\
    &=\sqrt{\det(\bA^\top \bA)}\int_{[0,1]^n}\left(\bc^\top \bA \by + \bc^\top\bb\right)^q d\by.
\end{align*}
If $\bA=u\mathbf{I}_d$ (which implies $n=d$) and $\bb=\bv_0$, then we have
\begin{align*}
    \int_{Q} f(\bx) d\bx &= u^d\int_{[0,1]^d}\left(u\bc^\top\by+\bc^\top\bv_0\right)^q d\by= u^{q+d}\int_{[0,1]^d}\left(\bc^\top\by+\frac{\bc^\top\bv_0}{u}\right)d\by. \tag*{\qed}
\end{align*}
\end{proof}

Suppose that $\bc>0$ and $\bv_0>0$; then $f(\bx)=(\bc^\top\bx)^q$ ($q>1$)
is convex on $\conv(Q_u\cup \{\mathbf{0}\})$. Because $f(\bx)$ is also $q$-homogeneous, we can use Corollary \ref{cor:homo_delta} and Proposition \ref{prop:power_linear} to obtain:

\begin{corollary}
For $f(\bx) := (\bc^\top \bx)^q (\bc>0, q\ge1)$ and $Q:=Q_u$, we have
\begin{align*}
  & \Delta(P^0,P)=\frac{q-1}{(d+2)(q+d+1)}u^{q+d}\int_{[0,1]^d} \left(\bc^\top \by+\frac{\bc^\top\bv_0}{u}\right)^q d\by.
\end{align*}
\end{corollary}

\begin{remark}
For $f(\bx) := (\bc^\top \bx)^q (\bc>0, q\ge1)$ and $Q:=Q_u$,
$\int_{[0,1]^d}(\bc^\top\by)^q d\by$ and $\bc^\top\bv_0$ are both positive constants independent of $u$.
Therefore, $\int_{u[0,1]^d} f(\bx) d\bx$ and $\Delta(P^0,P)$ both have the order of $O(u^{q+d})$ as $u$ tends to infinity.
\end{remark}

To better understand the cut-off ratio studied in the next section, we are also interested in analyzing the ratio between $\int_{[0,1]^n}\left(\bc^\top \by+\frac{\bc^\top\bv_0}{u}\right)^q d\by$ and $\left(\bc^\top\mathbf{1}+\frac{\bc^\top\bv_0}{u}\right)^q=\max_{\by\in[0,1]^d}\left(\bc^\top \by+\frac{\bc^\top\bv_0}{u}\right)^q$.
It is clear that the ratio is upper bounded by $1$.
A lower bound on the ratio independent of $\bc$ will work for our analysis in the next section.
To inspire the analysis for the lower bound and also to collect the tools to compute the ratio for a specific $\bc$, we first provide two methods to compute  $\int_{[0,1]^d}(\bc^\top\by)^q d\by$ when $q\in\mathbb{Z}_{>1}$: one is to use the multinomial theorem and iterated integration on monomials; the other is to use a triangulation of the hypercube. Then we provide a lower bound using the idea of triangulation and the results from the simplex case.
\begin{proposition}\label{prop:two_methods}
For $q\in\mathbb{Z}_{>1}$, we have
\begin{equation}\label{eqn:power_linear_form_direct_int}
    \int_{[0,1]^n} \left(\bc^\top\by\right)^q d\by=\sum_{\substack{\|\alpha\|_1=q,~\alpha\in\mathbb{Z}^n_{\ge0}}}\frac{q!c_1^{\alpha_1}\dots c_n^{\alpha_n}}{(\alpha_1+1)!\dots(\alpha_n+1)!}.
\end{equation}
Suppose that for any nonempty subset $S$ of $[n]$; $\sum_{j\in S}c_j\ne 0$, then we have
\begin{equation}\label{eqn:power_linear_form_triangulation_int}
    \int_{[0,1]^n} \left(\bc^\top\by\right)^q d\by =\sum_{(i_1,\dots,i_n)\in\Omega}\frac{q!}{(q+n)!}\sum_{j=0}^n \frac{\left(\bc_{(i_1,\dots,i_n)}^\top\bw_j\right)^{q+n}}{\prod_{k\ne j}\left(\bc_{(i_1,\dots,i_n)}^\top(\bw_j-\bw_k)\right)},
\end{equation}
where $\bc_{(i_1,\dots,i_n)}:=(c_{i_1},\dots,c_{i_n})$, $\bw_0:=\mathbf{0}$, $\bw_j:=\sum_{\ell=n+1-j}^n \be_\ell$, and $\Omega$ is the set of all permutations of $[n]$.
\end{proposition}
\begin{proof}
By the multinomial theorem, we decompose $(\bc^\top\by)^q$ and obtain
\begin{align*}
    \int_{[0,1]^n} \left(\bc^\top\by\right)^q d\by&= \sum_{\substack{\|\alpha\|_1=q,~\alpha\in\mathbb{Z}^n_{\ge0}}}\frac{q!c_1^{\alpha_1}\dots c_n^{\alpha_n}}{\alpha_1!\dots \alpha_n!}\int_{[0,1]^d}y_1^{\alpha_1}\dots y_n^{\alpha_n} d\by\\
    &= \sum_{\substack{\|\alpha\|_1=q,~\alpha\in\mathbb{Z}^n_{\ge0}}}\frac{q!c_1^{\alpha_1}\dots c_n^{\alpha_n}}{\alpha_1!\dots \alpha_n!}\prod_{j=1}^n\int_{0}^1y_j^{\alpha_j}d y_j\\
    &= \sum_{\substack{\|\alpha\|_1=q,~\alpha\in\mathbb{Z}^n_{\ge0}}}\frac{q!c_1^{\alpha_1}\dots c_n^{\alpha_n}}{(\alpha_1+1)!\dots(\alpha_n+1)!}.
\end{align*}
Another way to compute $\int_{[0,1]^n} \left(\bc^\top\by\right)^q d\by$,
which will lead to a more tractable formula, is to triangulate the hypercube $[0,1]^n$, for example, by  Kuhn's triangulation $\Delta_{i_1,\dots,i_n}:=\{\bx: 0\le x_{i_1}\le\dots \le x_{i_n}\le 1\}$, for any permutation $(i_1,\dots,i_n)$ of $[n]$.
Because of the assumption on $\bc$, we know that $\bc_{(i_1,\dots,i_n)}^\top(\bw_j-\bw_k)\ne 0$ for any $(i_1,\dots,i_n)\in\Omega$ and $j\ne k$.
By using the short formulae of Brion for simplices (\cite[Equation (3)]{xuleemulti}, \cite{brion1988points}), we obtain a simpler formula for the calculation.
\begin{align*}
    \int_{[0,1]^n} \left(\bc^\top\by\right)^q d\by &= \sum_{(i_1,\dots,i_n)\in\Omega}\int_{\Delta_{i_1,\dots,i_n}} \left(\bc^\top\by\right)^q d\by\\
    &=\sum_{(i_1,\dots,i_n)\in\Omega}\frac{q!}{(q+n)!}\sum_{j=0}^n \frac{\left(\bc_{(i_1,\dots,i_n)}^\top\bw_j\right)^{q+n}}{\prod_{k\ne j}\left(\bc_{(i_1,\dots,i_n)}^\top(\bw_j-\bw_k)\right)}. \tag*{\qed}
\end{align*}
\end{proof}
\begin{remark}\label{rem:all_one}
For $\bc=\mathbf{1}$, using \eqref{eqn:power_linear_form_triangulation_int}, we have
\begin{align*}
    \int_{[0,1]^n} \left(\mathbf{1}^\top\by\right)^q d\by &= \sum_{(i_1,\dots,i_n)\in\Omega}\frac{q!}{(q+n)!}\sum_{j=0}^n \frac{\left(\mathbf{1}^\top\bw_j\right)^{q+n}}{\prod_{k\ne j}\left(\mathbf{1}^\top(\bw_j-\bw_k)\right)}\\
    &=\frac{q!n!}{(q+n)!}\sum_{j=0}^n \frac{j^{q+n}}{\prod_{k\ne j}(j-k)}\\
    &=\frac{q!}{(q+n)!}\sum_{j=0}^n (-1)^{n-j} \binom{n}{j} j^{q+n}.
\end{align*}
And we obtain
\[
\int_{[0,1]^n} \left(\mathbf{1}^\top\by\right)^q d\by=\sum\limits_{(i_1,\dots,i_n)\in\Omega}\int_{\Delta_{i_1,\dots,i_n}} \left(\mathbf{1}^\top\by\right)^q d\by
\geq \frac{q!n!}{(q+n+1)!}\sum_{j=0}^n j^q,
\]
where the lower bound for $\int_{\Delta_{i_1,\dots,i_n}}\left(\mathbf{1}^\top \by\right)^q d\by$ comes from \cite[Lemma 4.5]{xuleemulti}.
\end{remark}
We cannot obtain easily from \eqref{eqn:power_linear_form_direct_int} a better lower bound than $\frac{q!n!}{(q+n+1)!}\sum_{j=0}^n j^q$.
So, instead, we use the idea of triangulation to obtain the following better lower bound on the ratio between $\int_{[0,1]^n}\left(\bc^\top \by+\frac{\bc^\top\bv_0}{u}\right)^q d\by$ and $\left(\bc^\top\mathbf{1}+\frac{\bc^\top\bv_0}{u}\right)^q$.

\begin{lemma}\label{lem:c}
If $q\ge 1$, $\bc> 0$, $\bv_0>0$ and $u>0$, then
$$
\int_{[0,1]^d} \left(\bc^\top \bx+ \frac{\bc^\top\bv_0}{u}\right)^q d\bx \ge \frac{\Gamma(q+1)d!}{\Gamma(q+d+1)}\frac{\sum_{j=1}^{d}j^q}{d^q}\left(\bc^\top\mathbf{1}+\frac{\bc^\top\bv_0}{u}\right)^q.
$$
\end{lemma}
\begin{proof}
We apply  Kuhn's triangulation and use \cite[Lemma 4.5]{xuleemulti} to lower bound the integral over $\Delta_{i_1,\dots,i_d}:=\{\bx: 0\le x_{i_1}\le\dots \le x_{i_d}\le 1\}$, for any permutation $(i_1,\dots,i_d)$ of $[n]$. Let $\Omega$ be the set of all permutations of $[n]$.
\begin{align*}
&\quad\int_{[0,1]^d} \left(\bc^\top \bx+ \frac{\bc^\top\bv_0}{u}\right)^q d\bx=\int_{\frac{\bv_0}{u}+[0,1]^d} \left(\bc^\top \bx\right)^q d\bx \\
&=\sum_{(i_1,\dots,i_d)\in\Omega}\int_{\frac{\bv_0}{u}+\Delta_{i_1,\dots,i_d}} (\bc^\top\bx)^q d\bx\\
&\ge \sum_{(i_1,\dots,i_d)\in\Omega}\frac{\Gamma(q+1)}{\Gamma(q+d+1)}\left(\sum_{j=1}^{d}\left(\sum_{k=j}^d c_{i_k}+\frac{\bc^\top\bv_0}{u}\right)^q+\left(\frac{\bc^\top\bv_0}{u}\right)^q\right)\\
&= \frac{\Gamma(q+1)}{\Gamma(q+d+1)}\left(d!\left(\frac{\bc^\top\bv_0}{u}\right)^q+\sum_{j=1}^{d}\sum_{(i_1,\dots,i_d)\in\Omega}\left(\sum_{k=j}^d c_{i_k}+\frac{\bc^\top\bv_0}{u}\right)^q\right)\\
&\ge \frac{\Gamma(q+1)}{\Gamma(q+d+1)}\left(\sum_{j=1}^{d}\sum_{(i_1,\dots,i_d)\in\Omega}\left(\sum_{k=j}^d c_{i_k}+\frac{d-j+1}{d}\frac{\bc^\top\bv_0}{u}\right)^q\right)\\
&\ge \frac{\Gamma(q+1)d!}{\Gamma(q+d+1)}\sum_{j=1}^{d}\left(\frac{d-j+1}{d}\right)^q\left(\bc^\top\mathbf{1}+\frac{\bc^\top\bv_0}{u}\right)^q,
\end{align*}
where the first inequality follows from \cite[Lemma 4.5]{xuleemulti}, the second to last inequality comes from $1\ge \frac{d-j+1}{d}$ for all $j=1,\dots,d$, and the last inequality follows from the power mean inequality
$(\frac{1}{m}\sum_{j=1}^m x_j^q)^{1/q}\ge \frac{1}{m}\sum_{j=1}^m x_j$ for $\bx>0$ and $q>1$.
\qed
\end{proof}
\begin{remark}
Lemma \ref{lem:c} shows the ratio between $\int_{[0,1]^n}\left(\bc^\top \by+\frac{\bc^\top\bv_0}{u}\right)^q d\by$ and $\left(\bc^\top\mathbf{1}+\frac{\bc^\top\bv_0}{u}\right)^q=\max_{\by\in[0,1]^d}\left(\bc^\top \by+\frac{\bc^\top\bv_0}{u}\right)^q$ has a positive lower bound, which is the key observation to prove Theorem \ref{cx_rat}.

Lemma \ref{lem:c} is not tight. For example, if asymptotically $\bc := c_1 \be_{1}$, $\bv_0:=0$, we have
$$\int_{[0,1]^d} \left(\bc^\top \bx+ \frac{\bc^\top\bv_0}{u}\right)^q d\bx
=\frac{c_1^q}{q+1},$$
and the lower bound is
$$
\frac{\Gamma(q+1)d!}{\Gamma(q+d+1)}\frac{\sum_{j=1}^{d}j^q}{d^q}c_1^q \sim O\left(\frac{1}{q^d}\right)c_1^q.
$$
\end{remark}

\subsection{A class of exponential functions}

We consider $f(\bx) := e^{\bc^\top \bx}-1$. Notice that $f(\mathbf{0})=0$, but $f(\bx)$ is not $q$-homogeneous. We first consider the computations of   $\int_Q f(\bx) d\bx$ and\break  $\int_0^1 z^d\int_Q f(z\bx)d\bx dz$, which are the key parts to compute the volumes of $P(f,Q)$ and $P^0(f,Q)$, respectively.

\begin{proposition}\label{prop:pers_lower_bound_int}
For $Q:=
\left\{\bx=\bA\by+\bb\in\R^d ~:~ \by\in[0,1]^n\right\}$,
where $\bA\in\R^{d\times n}$, $\rank \bA =n\le d$, and $f(\bx) := e^{\bc^\top \bx}-1$, we have
\begin{align*}
    \int_{Q} f(\bx) d\bx  &= \sqrt{\det\left(\bA^\top \bA\right)}e^{\bc^\top \bb}\prod_{j=1}^n \frac{e^{\bc^\top\bA_{:,j}}-1}{\bc^\top \bA_{:,j}} - \sqrt{\det\left(\bA^\top \bA\right)},
\end{align*}
where $\bA_{:, j}$ is the $j$-th column of $\bA$.
In particular, when $Q:=Q_u$ ($\bA=u\mathbf{I}_d$, $\bb=\bv_0$), a unit hypercube scaled by $u$ and translated by $\bv_0$, then $\int_{Q} f(\bx) d\bx = e^{\bc^\top\bv_0}\prod_{j=1}^d \frac{e^{uc_j}-1}{c_j}-u^d$.
\end{proposition}
\begin{proof}
We have
\begin{align*}
    \int_{Q} f(\bx) d\bx  &= \sqrt{\det(\bA^\top \bA)}\int_{[0,1]^n} f(\bA \by+\bb)d\by\\
    &=\sqrt{\det(\bA^\top \bA)}e^{\bc^\top \bb}\int_{[0,1]^n} e^{\bc^\top\bA \by} d\by - \sqrt{\det(\bA^\top \bA)}.
\end{align*}
Then we iteratively integrate on the $y_j$, and obtain
\begin{align*}
    \int_{Q} f(\bx) d\bx
    &=\sqrt{\det(\bA^\top \bA)}e^{\bc^\top \bb}\prod_{j=1}^n \frac{e^{\bc^\top\bA_{:,j}}-1}{\bc^\top \bA_{:,j}} - \sqrt{\det(\bA^\top \bA)}.
\end{align*}
Thus, $\int_{Q_u}f(\bx)d\bx = e^{\bc^\top\bv_0}\prod_{j=1}^d \frac{e^{uc_j}-1}{c_j}-u^d$.
\qed\end{proof}

As in Proposition \ref{prop:two_methods} for $\left(\bc^\top\bx\right)^q$, another way to compute $\int_{[0,1]^n} e^{\bc^\top\bA \by} d\by$ is to triangulate the hypercube $[0,1]^n$.
However, the triangulation method does not provide us with a simpler formula as in Remark \ref{rem:all_one} or a lower bound as in Lemma \ref{lem:c}.
In the Appendix, we verify that the result obtained from the triangulation method is the same under a simple genericity assumption.

Unlike the case of homogeneous functions,
we need a further integration over $z$ to compute $ \int_0^1 z^d\int_Q f(z\bx)d\bx dz$ for the na\"{i}ve relaxation. By Proposition \ref{prop:pers_lower_bound_int}, we can obtain the following results.
\begin{proposition}\label{prop:naive_exp}
For $Q:=\{\bx=\bA\by+\bb\in\R^d ~:~ \by\in[0,1]^n\}$,
where $\bA\in\R^{d\times n}$, $\rank \bA =n\le d$, and $f(\bx) := e^{\bc^\top \bx}-1$, we have
\begin{align*}
     &
     \int_0^1 z^d\int_Q f(z\bx)d\bx dz = \int_{0}^{1} z^d \sqrt{\det\left(\bA^\top \bA\right)}e^{z\bc^\top \bb}\prod_{j=1}^n \frac{e^{z\bc^\top\bA_{:,j}}-1}{z\bc^\top \bA_{:,j}} dz -  \frac{\sqrt{\det\left(\bA^\top \bA\right)}}{d+1}.
\end{align*}
In particular, when $Q:=Q_u$ ($\bA:=u\mathbf{I}_d$, $\bb:=\bv_0$), then
$$
\int_0^1 z^d\int_Q f(z\bx)d\bx dz = \frac{1}{\prod_{j=1}^d c_j}\int_0^1 e^{z\bc^\top\bv_0}\prod_{j=1}^d\left(e^{zuc_j}-1\right)dz - \frac{u^d}{d+1}.
$$
\end{proposition}

\begin{corollary}\label{cor:diff_exp}
For $f(\bx) := e^{\bc^\top \bx}-1$ and $Q:=Q_u$, we have
\begin{align*}
  \Delta(P^0,P)~=~&
    \frac{1}{\prod_{j=1}^d c_j}\left(\frac{e^{\bc^\top\bv_0}\prod_{j=1}^d\left(e^{uc_j}-1\right)}{d+2} - \int_{0}^1 e^{z\bc^\top\bv_0}\prod_{j=1}^d\left(e^{zuc_j}-1\right)dz\right)\\
    &\quad +\frac{u^d}{(d+1)(d+2)}.
\end{align*}
\end{corollary}

\section{Cut-off ratio}\label{sec:cut_off}

As in \cite{xuleemulti}, we consider the asymptotic behavior of the cut-off ratio $\frac{\Delta(P^0,P)}{\vol(P^0(f,Q))}$ on the scaled box $Q_u$ to compare the two relaxations.
We extend the results of the different asymptotic behaviors for two classes of functions from the simplex case to the box case: (1) For the exponential function, the na\"{i}ve relaxation is close to the perspective relaxation in terms of the volume. (2) For the power of linear forms, the na\"{i}ve relaxation has a positive gap compared with the perspective relaxation.
Furthermore, we also obtain a sufficient condition for a nonnegative convex function $f$ to have its na\"{i}ve relaxation approaching its perspective relaxation asymptotically. This extends the results to a wider range of functions.

Although the cut-off amount $\Delta(P^0,P)$ is independent of the upper bound $\mu$, the denominator $\vol(P^0(f,Q))$ depends on $\mu$.
If the upper bound is very large, then $\vol(P^0(f,Q))$ will increase significantly, and the cut-off ratio between the two relaxations will be small. To ensure a fair comparison, we want to find a small concave upper bound to maintain the convexity of the relaxation and avoid the irrelevant parts when calculating the cut-off ratio.
In the simplex case, the smallest concave upper bound is a hyperplane that passes through all the vertices. However, the situation is more complex for the box case.
Thus, before we proceed with generalizing our findings, we first analyze the asymptotic behavior of the cut-off ratio against two natural upper bounds. Our analysis reveals that their asymptotic behavior differs only by a constant factor.
We also use $P^0(f,Q,\mu)$ to represent the na\"{i}ve relaxation $P^0(f,Q)$ to emphasize the dependence on the concave upper bound $\mu$.

One natural bound is the smallest concave upper bound for $f(x)$.
The \emph{concave envelope} of $f(x)$, denoted as $\mathrm{conc}(f)$, is defined as the smallest concave upper bound for $f(x)$ \cite[Definition 2.1]{tawarmalaniExplicitConvexConcave2013}. A function $f(x): S\mapsto \mathbb{R}$ is said to be \emph{supermodular} if $f(\bx\vee \by) + f(\bx \wedge \by)\ge f(\bx)+f(\by)$ for all $\bx,\by\in S$, where $\bx\vee\by$ and $\bx\wedge\by$ denotes the elementwise maximum and minimum of $\bx$ and $\by$ (see \cite{topkis1998supermodularity}, for example).

\begin{theorem}{\cite[Theorem 3.3, Lemma 3.2, Corollary 2.7]{tawarmalaniExplicitConvexConcave2013}}
If $f:[0,1]^d \mapsto \mathbb{R}$ is supermodular when restricted to $\{0,1\}^d$, then the concave envelope of $f$ over $[0,1]^d$ is given by the piecewise linear function through Kuhn's triangulation $\Delta_{i_1,\dots,i_d}:=\{\bx: 0\le x_{i_1}\le\dots \le x_{i_d}\le 1\}$, for any permutation $(i_1,\dots,i_d)$ of $[d]$, i.e.,
$$
\nu(\bx) = \bg_{i_1,\dots,i_d}^\top \bB_{i_1,\dots,i_d}^{-1} \bx + f(\mathbf{0}), \quad\forall \bx\in \Delta_{i_1,\dots,i_n}\,,
$$
where
\begin{align*}
&\bg_{i_1,\dots,i_d}^\top := [f(\bw_{i_1})-f(\mathbf{0}),\dots,f(\bw_{i_d})-f(\mathbf{0})],\\ &\bB:=[\bw_{i_1},\dots,\bw_{i_d}], \\
&\bw_{i_k} :=\textstyle \sum_{j=1}^k \be_{i_j} \mbox{ for } k\in[d].
\end{align*}
\end{theorem}
\begin{corollary}\label{cor:triang_u}
For $f(\bx) = e^{\bc^\top \bx}-1 (\bc\ge0)$ or $f(\bx)=(\bc^\top \bx)^q (q>1, \bc\ge0)$, the concave envelope of $f$ over $Q_u=\bv_0+u[0,1]^d$ ($\bv_0>0,u>0$) is given by the piecewise linear function through the triangulation $\bv_0+u\Delta_{i_1,\dots,i_d}$, where $\Delta_{i_1,\dots,i_d}:=\{\bx: 0\le x_{i_1}\le\dots \le x_{i_d}\le 1\}$, for any permutation $(i_1,\dots,i_d)$ of $[d]$, i.e.,
$$
\mathrm{conc}(f)(\bx) = \bg_{i_1,\dots,i_d}^\top \bB_{i_1,\dots,i_d}^{-1} (\bx-\bv_0) + f(\bv_0), \quad\forall \bx\in\bv_0+u \Delta_{i_1,\dots,i_n},
$$
where
\begin{align*}
&\bg_{i_1,\dots,i_d}^\top := [f(\bw_{i_1})-f(\bv_0),\dots,f(\bw_{i_d})-f(\bv_0)],\\ &\bB:=[\bw_{i_1}-\bv_0,\dots,\bw_{i_d}-\bv_0], \\
&\bw_{i_k} :=\textstyle \bv_0+u\sum_{j=1}^k \be_{i_j} \mbox{ for } k\in[d].
\end{align*}
\end{corollary}
\begin{proof}
It is easy to verify that both $f\left(\bv_0+u\bx\right)=e^{u\bc^\top\bx+\bc^\top\bv_0}-1$ and\break $f\left(\bv_0+u\bx\right)=(u\bc^\top\bx+\bc^\top\bv_0)^q$ are supermodular on $[0,1]^d$ by checking that $\frac{\partial^2 f}{\partial x_i \partial x_j}\ge0$ for $i\ne j$ (see \cite[Section 2.6.1]{topkis1998supermodularity}).
\qed
\end{proof}

The supermodularity of $e^{u\bc^\top \bx+\bc^\top\bv_0}-1$ $(\bc\ge0,\bv>0)$ and $(u\bc^\top \bx+\bc^\top\bv_0)^q$ $(q>1, \bc\ge0,\bv>0)$ can also be seen from the fact that a composition of an increasing convex function and an increasing supermodular function (here we use a linear function) is  supermodular.

Next, we consider two natural upper bound functions $\mu$ and compare the asymptotic behavior of the cut-off ratio for $f(\bx):=e^{\bc^\top \bx}-1$ $(\bc> 0)$ on a scaled box $Q_u$: (1) $\mu$ is a best constant upper bound $F:=\max_{\bx\in Q}\{f(\bx)\}$; (2) $\mu$ is $\mathrm{conc}(f)$. In both cases, we will demonstrate that the cut-off ratio goes to 0 like a constant times $u^{-d}$, with the constant for the trivial upper bound exceeding the constant for the optimal concave upper bound by a factor of exactly $d+1$.
In fact, we can conclude that for any upper bounding function in between these two,
the asymptotic behavior of the cut-off ratio  is between these
two asymptotic behaviors.
This implies that for the exponential function $f(\bx):=e^{\bc^\top \bx}-1$ $(\bc> 0)$, the na\"{i}ve relaxation is close to the perspective relaxation asymptotically.
\begin{lemma}\label{lem:two_mu}
Let $f(\bx) := e^{\bc^\top \bx}-1$ $(\bc> 0)$, $Q_u:=\bv_0+u[0,1]^d$ ($\bv_0>0,u>0$), $F:=\max_{\bx\in Q_u} \{f(\bx)\} = f(\bv_0+u\mathbf{1})= e^{\bc^\top\bv_0}e^{u\bc^\top\mathbf{1}}-1$. Then
\begin{align*}
    &\lim_{u\rightarrow\infty} \frac{\vol(P^0(f,Q_u,F))}{\vol(P^0(f,Q_u,\mathrm{conc}(f)))} = d+1.
\end{align*}
\end{lemma}

\begin{proof}
By Theorem \ref{thm:naivevol}, we need to compute $\frac{1}{d+2}\int_{Q_u}\mu(\bx)d\bx$ for the two upper bounds.
For the upper bound $F:=e^{\bc^\top\bv_0}e^{u\bc^\top\mathbf{1}}-1$, we have
$$
\frac{1}{d+2}\int_{Q_u} F d\bx = \frac{u^d F}{d+2}.
$$
Let $\nu:=\mathrm{conc}(f)$, and let $\Omega$ be the set of all permutations of $[d]$.
Using the triangulation in Corollary \ref{cor:triang_u}, we have
$$\frac{1}{d+2}\int_{Q_u} \nu(\bx) d\bx=\frac{1}{d+2}\sum_{(i_1,\dots,i_d)\in\Omega}\int_{\bv_0+u\Delta_{i_1,\dots,i_d}} \nu(\bx) d\bx.
$$
By Lemma 2.1 in \cite{xuleemulti}, we have
\begin{align*}
 &\frac{1}{d+2}\sum_{(i_1,\dots,i_d)\in\Omega}\int_{\bv_0+u\Delta_{i_1,\dots,i_d}} \nu(\bx) d\bx\\
 \quad =&\frac{1}{d+2}\sum_{(i_1,\dots,i_d)\in\Omega}\frac{u^d}{(d+1)!}\left(f(\bv_0)+\sum_{k=1}^d f(\bv_0+u\sum_{j=1}^k\be_{i_j})\right)\\
 \quad =&\frac{u^d}{(d+2)!}\sum_{S\subseteq[d]}\abs{S}!(d-\abs{S})!f(\bv_0+u\sum_{j\in S}\be_{j}).
\end{align*}
Therefore, due to $\lim\limits_{u\rightarrow\infty}{f\left(\bv_0+ u\sum_{j\in S}\be_j\right)}/{F}=0$ for every $S\subsetneq[d]$, we have
\begin{align*}
&\quad \quad\frac{1}{d+2}\int_{Q_u} \nu(\bx) d\bx\\
 &\quad =\frac{u^d F}{(d+2)(d+1)}+\frac{u^d}{(d+2)!}\sum_{S\subsetneq[d]}\abs{S}!(d-\abs{S})!f(\bv_0+u\sum_{j\in S}\be_{j})\sim \frac{u^d F}{(d+2)(d+1)}.
\end{align*}
By Proposition \ref{prop:naive_exp}, we collect the highest-order term as $u$ tends to infinity and obtain
\begin{align*}
&\int_0^1 z^d\int_{Q_u} f(z\bx)d\bx dz=\frac{1}{\prod_{j=1}^d c_j}\int_0^1 e^{z\bc^\top\bv_0}\prod_{j=1}^d\left(e^{zuc_j}-1\right)dz - \frac{u^d}{d+1} \sim \frac{F}{u^d\prod_{j=1}^d c_j^2}.
\end{align*}
Thus, the integral $\int_0^1 z^d\int_{Q_u} f(z\bx)d\bx dz$ can be neglected when computing the volume of the na\"{i}ve relaxation as $u$ tends to infinity:
\begin{align*}
\lim_{u\rightarrow \infty}\frac{\int_0^1 z^d\int_{Q_u} f(z\bx)d\bx dz}{F} = 0.
\end{align*}
Then by Theorem \ref{thm:naivevol} and the above result, we have
\begin{align*}
&\lim_{u\rightarrow\infty} \frac{\vol(P^0(f,Q_u,F))}{\vol(P^0(f,Q_u,\mathrm{conc}(f)))}=\lim_{u\rightarrow\infty} \frac{\frac{1}{d+2}\int_{Q_u} F d\bx}{\frac{1}{d+2}\int_{Q_u} \nu(\bx) d\bx}\\
=&\lim_{u\rightarrow \infty}\frac{\frac{u^d F}{d+2}}{\frac{u^d F}{(d+2)(d+1)}}= d+1.    \tag*{\qed}
\end{align*}
\end{proof}
\begin{theorem}\label{exprat}
Let $f(\bx) := e^{\bc^\top \bx}-1$ $(\bc> 0)$, $Q_u:=\bv_0+u[0,1]^d$ ($\bv_0>0,u>0$), $F:=\max_{\bx\in Q_u} \{f(\bx)\} = f(\bv_0+u\mathbf{1})= e^{\bc^\top\bv_0}e^{u\bc^\top\mathbf{1}}-1$. Then
\begin{align*}
    &\lim_{u\rightarrow\infty} u^d\cdot \frac{\Delta(P^0,P)}{\vol(P^0(f,Q_u,\mathrm{conc}(f)))} = \lim_{u\rightarrow\infty} u^d\cdot\frac{(d+1)\Delta(P^0,P)}{\vol(P^0(f,Q_u,F))}=\frac{d+1}{\prod_{j=1}^d c_j}.
\end{align*}
\end{theorem}
\begin{proof}
By Proposition \ref{prop:pers_lower_bound_int} and \ref{prop:naive_exp}, we collect the highest order term as $u$ tends to infinity and obtain
\begin{align*}
&\int_0^1 z^d\int_{Q_u} f(z\bx)d\bx dz=\frac{1}{\prod_{j=1}^d c_j}\int_0^1 e^{z\bc^\top\bv_0}\prod_{j=1}^d(e^{zuc_j}-1)dz - \frac{u^d}{d+1} \sim \frac{F}{u^d\prod_{j=1}^d c_j^2},\\
&\int_{Q_u}f(\bx)d\bx = e^{\bc^\top\bv_0}\frac{\prod_{j=1}^d(e^{uc_j}-1)}{\prod_{j=1}^d c_j}-u^d\sim \frac{F}{\prod_{j=1}^d c_j}.
\end{align*}
Thus,
\begin{align*}
\lim_{u\rightarrow \infty}\frac{\int_0^1 z^d\int_{Q_u} f(z\bx)d\bx dz}{\int_{Q_u}f(\bx)d\bx} = 0.
\end{align*}
Then by Corollary \ref{cor:diff_exp} and the above result, we have
\begin{align*}
    \Delta(P^0,P)
    &\sim \frac{1}{d+2}\int_{Q_u} f(\bx)d\bx \sim \frac{F}{(d+2)\prod_{j=1}^d c_j}.
\end{align*}
Therefore, by Lemma \ref{lem:two_mu},
\begin{align*}
    &\lim_{u\rightarrow\infty} u^d\cdot \frac{\Delta(P^0,P)}{\vol(P^0(f,Q_u,\mathrm{conc}(f)))} = \lim_{u\rightarrow\infty} u^d\cdot\frac{(d+1)\Delta(P^0,P)}{\vol(P^0(f,Q_u,F))}=\frac{d+1}{\prod_{j=1}^d c_j}. \tag*{\qed}
\end{align*}
\end{proof}
Theorem \ref{exprat} extends \cite[Theorem 4.13(b)]{xuleemulti} from a scaled simplex to a scaled hypercube.

Now, we analyze the asymptotic behavior of the cut-off ratio for $f(\bx):=(\bc^\top \bx)^q$ $(\bc> 0, q>1)$ on a scaled box $Q_u$\thinspace, for the best constant upper bound $F:=\max_{\bx\in Q}\{f(\bx)\}$. Because $\vol(P^0(f,Q_u,F))\ge\vol(P^0(f,Q_u,\mathrm{conc}(f)))$, we know that the asymptotic cut-off ratio for the upper bound $\mathrm{conc}(f)$ is lower bounded by the cut-off ratio for the upper bound $F$.
We will demonstrate that the cut-off ratio has a positive lower bound as $u$ tends to infinity, which has  similar behavior to the simplex case in \cite{xuleemulti}.
This positive gap implies that for the power of linear forms $f(\bx):=(\bc^\top \bx)^q$ $(\bc> 0, q>1)$, the na\"{i}ve relxation is not close to the perspective relaxation asymptotically.

\begin{theorem}\label{cx_rat}
Let $f(\bx) := (\bc^\top \bx)^q$ $(\bc> 0, q>1)$, $Q_u:=\bv_0+u[0,1]^d$ ($\bv_0>0,u>0$), $F:=\max_{\bx\in Q_u} \{f(\bx)\} = f(u\mathbf{1})= (u\bc^\top\mathbf{1}+\bc^\top\bv_0)^q$. Then
\begin{align*}
    &\lim_{u\rightarrow\infty}\frac{\Delta(P^0,P)}{\vol(P^0(f,Q_u,F))}\ge\frac{(q-1)\frac{\Gamma(q+1)d!}{\Gamma(q+d+1)}\frac{\sum_{j=1}^{d}j^q}{d^q}}{(q+d+1)-(d+2)\frac{\Gamma(q+1)d!}{\Gamma(q+d+1)}\frac{\sum_{j=1}^{d}j^q}{d^q}}.
\end{align*}
\end{theorem}
\begin{proof}
$f(\bx):=(\bc^\top \bx)^q$ $(\bc> 0, q>1)$  is $q$-homogeneous.
By Theorem \ref{thm:naivevol}, Lemma \ref{lem:naivevol_qhomo}, and Corollary \ref{cor:homo_delta}, we obtain
\begin{align*}
    &\frac{\Delta(P^0,P)}{\vol(P^0(f,Q_u,F))}=\frac{\frac{(q-1)}{(d+2)(q+d+1)}\int_{Q_u} (\bc^\top\bx)^q d\bx}{\frac{1}{d+2}\int_{Q_u} F d\bx-\frac{1}{q+d+1}\int_{Q_u} (\bc^\top \bx)^q d\bx}\\
    &=\frac{\frac{(q-1)u^{q+d}}{(d+2)(q+d+1)}\int_{[0,1]^d} (\bc^\top\bx+ \frac{\bc^\top\bv_0}{u})^q d\bx}{\frac{u^{q+d} (\bc^\top\mathbf{1}+ \frac{\bc^\top\bv_0}{u})^q}{d+2}-\frac{u^{q+d}}{q+d+1}\int_{[0,1]^d} (\bc^\top \bx+ \frac{\bc^\top\bv_0}{u})^q d\bx}\\
    &=\frac{(q-1)\int_{[0,1]^d} (\frac{\bc^\top \bx+ \frac{\bc^\top\bv_0}{u}}{\bc^\top\mathbf{1}+ \frac{\bc^\top\bv_0}{u}})^q d\bx}{(q+d+1)-(d+2)\int_{[0,1]^d} (\frac{\bc^\top \bx+ \frac{\bc^\top\bv_0}{u}}{\bc^\top\mathbf{1}+ \frac{\bc^\top\bv_0}{u}})^q d\bx}.
\end{align*}

Then by Lemma \ref{lem:c}, we have
\begin{align*}
\frac{\Delta(P^0,P)}{\vol(P^0(f,Q_u,F))} & =\frac{(q-1)\int_{[0,1]^d} (\frac{\bc^\top \bx+ \frac{\bc^\top\bv_0}{u}}{\bc^\top\mathbf{1}+ \frac{\bc^\top\bv_0}{u}})^q d\bx}{(q+d+1)-(d+2)\int_{[0,1]^d} (\frac{\bc^\top \bx+ \frac{\bc^\top\bv_0}{u}}{\bc^\top\mathbf{1}+ \frac{\bc^\top\bv_0}{u}})^q d\bx}\\
&\ge \frac{(q-1)\frac{\Gamma(q+1)d!}{\Gamma(q+d+1)}\frac{\sum_{j=1}^{d}j^q}{d^q}}{(q+d+1)-(d+2)\frac{\Gamma(q+1)d!}{\Gamma(q+d+1)}\frac{\sum_{j=1}^{d}j^q}{d^q}}. \tag*{\qed}
\end{align*}
\end{proof}

Next, as a first step to extend the results to a broader family of functions, based on the proofs of Theorem \ref{exprat}, we derive a sufficient condition for a nonnegative convex function $f$ to have an asymptotic cut-off ratio $0$ on the scaled box $Q_u$. We also give some examples besides the exponential function.

\begin{proposition}\label{prop:sufficient_condition}
Suppose that $f(\bx)$ is a nonnegative convex function on $\mathbb{R}^d$ satisfying $f(\mathbf{0})=0$. Suppose also that  $\lim_{u\rightarrow \infty}\frac{\int_{[0,1]^d} f(\bv_0+u\bx) d\bx}{\int_{[0,1]^d} \mu(\bv_0+u\bx) d\bx}=0$, where $\mu$ is a concave upper bound function for $f$ on the scaled box $Q_u:=\bv_0+ u[0,1]^d$. Then
$$
\lim_{u\rightarrow\infty}\frac{\Delta(P^0,P)}{\vol(P^0(f,Q_u,\mu))}=0.
$$
In particular, if $\mu(\bx)=\max_{Q_u}f(\bx)$ is the best constant upper bound, then the sufficient condition is equivalent to $\lim_{u\rightarrow \infty}\int_{[0,1]^d} \frac{f(\bv_0+u\bx)}{\mu(\bv_0+u\bx)} d\bx=0$.
\end{proposition}
\begin{proof}
Because $f$ is nonnegative, we have $\int_0^1 z^d\int_{Q_u} f(z\bx)d\bx dz\ge 0$. Thus
\begin{align*}
    0\le\frac{\Delta(P^0,P)}{\vol(P^0(f,Q_u,\mu))}&= \frac{\frac{1}{d+2}\int_{Q_u} f(\bx) d\bx-\int_0^1 z^d\int_{Q_u} f(z\bx)d\bx dz}{\frac{1}{d+2}\int_{Q_u} \mu(\bx) d\bx-\int_0^1 z^d\int_{Q_u} f(z\bx)d\bx dz}\\
    &\le \frac{\frac{1}{d+2}\int_{Q_u} f(\bx) d\bx}{\frac{1}{d+2}\int_{Q_u} \mu(\bx) d\bx}\\
    &=\frac{\frac{1}{d+2}\int_{[0,1]^d} f(\bv_0+u\bx) d\bx}{\frac{1}{d+2}\int_{[0,1]^d} \mu(\bv_0+u\bx) d\bx}.
\end{align*}
Therefore,
\begin{align*}
\lim_{u\rightarrow \infty}\frac{\Delta(P^0,P)}{\vol(P^0(f,Q_u,\mu))} = 0. \tag*{\qed}
\end{align*}
\end{proof}

\begin{example}
For $f(\bx):=e^{\bc^\top \bx}-1$ with $\bc>0$, and $\mu(\bx):=\max_{Q_u} f(\bx) = e^{u\bc^\top\mathbf{1}+\bc^\top\bv_0}-1$, we have
\begin{align*}
&\lim_{u\rightarrow \infty}\int_{[0,1]^d} \frac{f(\bv_0+u\bx)}{\mu(\bv_0+u\bx)} d\bx
~=~\lim_{u\rightarrow \infty}\int_{[0,1]^d} \frac{e^{u\bc^\top \bx+\bc^\top\bv_0}-1}{e^{u\bc^\top\mathbf{1}+\bc^\top\bv_0}-1}d\bx\\
&\quad =~\lim_{u\rightarrow \infty}\int_{[0,1]^d} e^{-u\bc^\top(\mathbf{1}-\bx)}d\bx~=~0.
\end{align*}
\end{example}
Therefore, the asymptotic cut-off ratio is $0$ by Proposition \ref{prop:sufficient_condition}, which also follows from Theorem \ref{exprat}.

\begin{example}
We can also consider other superpolynomial functions, for example, $f(\bx) :=g(\bc^\top \bx)$, where $\bc>0$ and $g(t)=(t+1)^{\log(t+1)} - 1$, $t\ge0$.
Let $\mu(\bx) := \max_{Q_u}f(\bx)= g(u\bc^\top \mathbf{1}+\bc^\top\bv_0)$.

We can verify that
\begin{align*}
    \lim\limits_{u\rightarrow \infty}\frac{g(u\bc^\top\bx+\bc^\top\bv_0)}{g(u\bc^\top\mathbf{1}+\bc^\top\bv_0)}    &=\lim\limits_{u\rightarrow \infty}e^{(\log(u\bc^\top\bx))^2 - (\log(u\bc^\top\mathbf{1}))^2}\\
&= \lim\limits_{u\rightarrow \infty}e^{2\log (u)(\log(\bc^\top\bx)-\log(\bc^\top\mathbf{1})) + (\log(\bc^\top\bx))^2 - (\log(\bc^\top\mathbf{1}))^2 }\\
&=0, ~\text{for}~\bx\ne \mathbf{1}.
\end{align*}

Therefore, the condition can be verified by
\begin{align*}
    \lim_{u\rightarrow \infty}\int_{[0,1]^d} \frac{f(\bv_0+u\bx)}{\mu(\bv_0+u\bx)} d\bx &=\lim\limits_{u\rightarrow \infty}\int_{[0,1]^d}\frac{g(u\bc^\top\bx+\bc^\top\bv_0)}{g(u\bc^\top\mathbf{1}+\bc^\top\bv_0)} d\bx\\
    &=\int_{[0,1]^d}\lim\limits_{u\rightarrow \infty}\frac{g(u\bc^\top\bx+\bc^\top\bv_0)}{g(u\bc^\top\mathbf{1}+\bc^\top\bv_0)} d\bx\\
    &=0.
\end{align*}

By Proposition \ref{prop:sufficient_condition}, the cut-off ratio tends to 0.
\end{example}

\section{Concluding remarks}\label{sec:conc}
We lifted several results comparing the perspective and na\"{i}ve relaxations in the simplex-domain case \cite{xuleemulti} to the box-domain case.
\begin{itemize}
    \item We extended the computation of the volume of perspective and na\"{i}ve relaxations to general polytopes and connect it to integration. We also provided  results computing the integration of $f(\bx)$ for the zonotope case, which includes the box case.
    \item We compared two natural concave upper bounds, the constant upper bound and the concave-envelope upper  bound, and showed that the cut-off ratio has similar asymptotic behavior for these two bounds.
    \item Using the idea of triangulation, we analyzed the cut-off ratio on a scaled box for two important classes of functions, and showed different asymptotic behavior for the power of linear functions and a class of exponential functions.
    \item We gave a sufficient condition for a function $f(\bx)$ such that the cut-off ratio tends to $0$ as the scaled box becomes larger, which is the asymptotic behavior for a class of exponential functions.
\end{itemize}
For  future work, we believe that some technical improvements can be achieved, for example, the lower bound in Theorem \ref{cx_rat} and Lemma \ref{lem:c}.
It will be very interesting to understand the asymptotic behavior of the cut-off ratio in terms of more general classes of functions and domains. As a first step, we will need to understand the integration over a general polytope and the concave upper bound of a convex function beyond the supermodular function on a box case studied in this paper.
Another interesting direction is to explore the effects of triangulation over a box on these relaxations.

\begin{acknowledgements}
We gratefully acknowledge discussions with Zhongzhu Chen in regard to Lemma \ref{lem:c}.
\end{acknowledgements}

\section*{Declarations}
The authors declare that they have no competing interests.


\bibliographystyle{alpha}
\bibliography{perssimplex}   

\appendix

\section*{Appendix: Triangulation method for the integral of an exponential function on a hypercube}

Let $\tilde{\bc}^\top:=\bc^\top\bA$. In this appendix, we verify that the result obtained from the triangulation method to compute $\int_{[0,1]^n} e^{\tilde{\bc}^\top\by} d\by$ is the same as $\prod_{j=1}^n \frac{e^{\tilde{c}_j}-1}{\tilde{c}_j}$ in Proposition \ref{prop:pers_lower_bound_int} under a simple genericity assumption: for any nonempty subset $S$ of $[n]$, $\sum_{j\in S}\tilde{c}_j\ne 0$.
This assumption ensures that we can use the short formulae of Brion to compute over each simplex. With Kuhn's triangulation, we obtain
\begin{align*}
    \int_{[0,1]^n} e^{\tilde{\bc}^\top\by} d\by &= \sum_{(i_1,\dots,i_n)\in\Omega}\int_{\Delta_{i_1,\dots,i_n}} e^{\tilde{\bc}^\top\by} d\by\\
    &=\sum_{(i_1,\dots,i_n)\in\Omega}\sum_{j=0}^n \frac{e^{\tilde{\bc}_{(i_1,\dots,i_n)}^\top\bw_j}}{\prod_{k\ne j}\left(\tilde{\bc}_{(i_1,\dots,i_n)}^\top(\bw_j-\bw_k)\right)},
\end{align*}
where $\Delta_{i_1,\dots,i_n}:=\{\bx: 0\le x_{i_1}\le\dots \le x_{i_n}\le 1\}$, $\tilde{\bc}_{(i_1,\dots,i_n)}:=(\tilde{c}_{i_1},\dots,\tilde{c}_{i_n})$, $\bw_0:=\mathbf{0}$, $\bw_j:=\sum_{\ell=n+1-j}^n \be_\ell$, and $\Omega$ is the set of all permutations of $[n]$.

We are going to verify that
$$
\sum_{(i_1,\dots,i_n)\in\Omega}\sum_{j=0}^n \frac{e^{\tilde{\bc}_{(i_1,\dots,i_n)}^\top\bw_j}}{\prod_{k\ne j}\left(\tilde{\bc}_{(i_1,\dots,i_n)}^\top(\bw_j-\bw_k)\right)} = \prod_{j=1}^n \frac{e^{\tilde{c}_j}-1}{\tilde{c}_j},
$$
by comparing the coefficient of the term $e^{\tilde{c}_{j_1}+\tilde{c}_{j_2}+\dots+\tilde{c}_{j_k}}$ on both sides.
Because
$$
\prod_{j=1}^n \frac{e^{\tilde{c}_j}-1}{\tilde{c}_j} = \frac{1}{\prod_{j=1}^n\tilde{c}_j}\sum_{k=0}^n (-1)^{n-k}\sum_{1\le j_1<j_2<\dots<j_k\le n}e^{\tilde{c}_{j_1}+\tilde{c}_{j_2}+\dots+\tilde{c}_{j_k}},
$$
we know the coefficient on the right-hand side is $\frac{(-1)^{n-k}}{\prod_{j=1}^n\tilde{c}_j}$.

Before computing the coefficient on the left-hand side, we first prove that for $\ell := |T|$,
\begin{equation}\label{eqn:claim_1}
\sum_{(i_1,\dots,i_{\ell})\in\Omega(T)}\frac{1}{\tilde{c}_{i_{1}}(\tilde{c}_{i_{1}}+\tilde{c}_{i_{2}})\dots(\tilde{c}_{i_{1}}+\dots+\tilde{c}_{i_{\ell}})}=\frac{1}{\prod_{j\in T}\tilde{c}_j},
\end{equation}
by induction on $\ell$. The result is trivial when $\ell=1$. Suppose the result holds for $\ell-1\ge 1$, then we can compute by the inductive hypothesis for each fixed $i_{\ell}$\thinspace,
\begin{align*}
&\sum_{(i_1,\dots,i_{\ell})\in\Omega(T)}\frac{1}{\tilde{c}_{i_{1}}(\tilde{c}_{i_{1}}+\tilde{c}_{i_{2}})\dots(\tilde{c}_{i_{1}}+\dots+\tilde{c}_{i_{\ell}})}\\
&\quad = \sum_{j\in T} \frac{1}{\sum_{s\in T}\tilde{c}_s}\sum_{(i_1,\dots,i_{\ell})\in\Omega(T),i_{\ell}=j}\frac{1}{\tilde{c}_{i_{1}}(\tilde{c}_{i_{1}}+\tilde{c}_{i_{2}})\dots(\tilde{c}_{i_{1}}+\dots+\tilde{c}_{i_{\ell-1}})}\\
&\quad = \sum_{j\in T} \frac{1}{\sum_{s\in T}\tilde{c}_s} \frac{1}{\prod_{s\in T\setminus\{j\}}\tilde{c}_s}~\quad(\text{by the inductive hypothesis})\\
&\quad = \frac{1}{\prod_{s\in T}\tilde{c}_s}.
\end{align*}
Thus, \eqref{eqn:claim_1} holds.

Let $S:=\{j_1,j_2,\dots,j_k\}$, $S^c:=[n]\setminus S$. The coefficient of the term $e^{\tilde{c}_{j_1}+\tilde{c}_{j_2}+\dots+\tilde{c}_{j_k}}$ on the left-hand side is
\begin{align*}
 &\sum_{(i_1,\dots,i_n)\in\Omega, (i_1,\dots,i_{n-k})\in S^c}\frac{1}{\prod_{k\ne j}(\tilde{\bc}_{(i_1,\dots,i_n)}^\top(\bw_j-\bw_k))}\\
=&\sum_{(i_1,\dots,i_{n-k})\in\Omega(S_c)}\frac{1}{-\tilde{c}_{i_{n-k}}(-\tilde{c}_{i_{n-k-1}}-\tilde{c}_{i_{n-k}})\dots(-\tilde{c}_{i_{1}}-\dots-\tilde{c}_{i_{n-k}})}\\
&\cdot \sum_{(i_{n-k+1,\dots,i_{n}})\in\Omega(S)}\frac{1}{\tilde{c}_{i_{n-k+1}}(\tilde{c}_{i_{n-k+1}}+\tilde{c}_{i_{n-k+2}})\dots(\tilde{c}_{i_{n-k+1}}+\dots+\tilde{c}_{i_{n}})}\\
=&\frac{1}{(-1)^{n-k}\prod_{j\in S_c}\tilde{c}_j} \cdot \frac{1}{\prod_{j\in S}\tilde{c}_j}\\
=&\frac{(-1)^{n-k}}{\prod_{j\in[n]}\tilde{c}_j},
\end{align*}
where $\Omega(T)$ is the set of all permutations of $T$, and the penultimate equation follows from \eqref{eqn:claim_1}. Therefore, we have demonstrated that the results from both methods are the same.
\end{document}